\documentclass{article}
\usepackage{float}
\usepackage{url}
\usepackage{bm,amsmath,amsthm,amssymb,graphicx,color,tikz,pgfplots,tkz-fct}
\usetikzlibrary{intersections}
\usepgfplotslibrary{fillbetween}
\theoremstyle{plain} 
\newtheorem{thm}{Theorem}
 
\theoremstyle{definition}

\newcommand{\T}{\textsf{T}}

\title{Counter-terrorism analysis \\ using cooperative game theory}
\author{Sung Chan Choi\thanks{Department of Mathematics, University of Utah, 155 S. 1400 E., Salt Lake City, UT 84112, USA. e-mail: choi@math.utah.edu}}
\date{}

\pgfplotsset{compat=1.16}

\begin{document}
\maketitle

\begin{abstract}
Game theory has been applied in many fields of study, especially economics and political science.  Arce M. and Sandler (2005) analyzed counter-terrorism using non-cooperative game theory (the players are, for example, the US and the EU), which assumes that communication among the players is not allowed, or, if it is allowed, then there is no mechanism to enforce any agreement the players may make. The only solution in the non-cooperative setting would be a Nash equilibrium because the players adopt only self-enforcing strategies. Here we analyze counter-terrorism using cooperative game theory, because there are ways to communicate among the players and to make binding agreements; indeed, countries that oppose terrorism are closely connected to each other in many aspects such as economically and in terms of international politics. 
\end{abstract}

\section{Introduction}

Arce M. and Sandler (2005) classified counter-terrorism policies into preemption, no action, and deterrence. Preemption is a proactive policy in which terrorists and their assets are attacked to curb subsequent terrorist campaigns. It can protect all potential targets from terrorists. Deterrence comprises more defensive or passive counter-terrorism measures that include making technological barriers such as metal detectors or bomb-sniffing equipment at airports, fortifying potential targets, and securing borders. These defensive policies are intended to deter an attack by either making success more difficult or increasing the likelihood of negative consequences for the terrorists.

The reason why many countries facing terrorism are more inclined to choose the deterrence policy rather than preemption, despite the greater social gain using preemption, is that the famous ``prisoner's dilemma'' is hidden in the game, as we will point out below.

Since preemption can protect all potential targets, it provides public benefits. In contrast, deterrence imposes public costs because it can deflect the attack to relatively less-guarded targets. We assume that each preemption gives a public benefit of 4 for player 1 and player 2 at a private cost of 6 to the player who uses preemption.  Comparing with deterrence, it imposes a public cost of 4 on both the deterrer and the other because the nondeterrer suffers the deflection costs of being the target of choice, and it provides private gains of 6 to the only deterrer motivated by greater amount of gain than the cost.  The payoff bimatrix from Arce M. and Sandler (2005), in which the row player is player 1 (e.g., the US) and the column player is player 2 (e.g., the EU), is given by
\begin{equation}\label{bimatrix}
\bordermatrix{
 & \text{Preempt} & \text{Status Quo} & \text{Deter}\cr
\text{Preempt} & (2,2) & (-2,4) & (-6,6)\cr
\text{Status Quo} & (4,-2) & (0,0) & (-4,2)\cr
\text{Deter} & (6,-6) & (2,-4) & (-2,-2)}.
\end{equation}
  
\begin{itemize}
\item {(Preempt, Preempt)}\\
The players give a public benefit $(=4)$ to each other so both can take a total benefit of 8 and they pay a private cost $(=6)$ respectively. Therefore each payoff is equal to $2$ $(=4+4-6)$.
 
\item {(Preempt, Status Quo) or (Status Quo, Preempt)}\\
The preemptor can gain public benefit $(=4)$ from himself preempting but pay private cost $(=6)$. Hence his payoff will be $-2$ $(=4-6)$. However, the player adopting the status quo can only get the benefit $(=4)$ which the preemptor makes without any cost. So the payoff to the player doing nothing is $4$ $(=4-0)$.

\item {(Preempt, Deter) or (Deter, Preempt)}\\
The payoff to the preemptor is $-6$ $(=4-6-4)$ because he can enjoy his own public benefit $(=4)$ but has to pay a private cost $(=6)$ and a public cost $(=4)$ raised by the deterrer together. On the other hand, since the deterrer can attain a private benefit $(=6)$ from himself and also a public benefit $(=4)$ by the preemptor but has only to pay the public cost $(=4)$ raised by his deterring, the payoff to the deterrer is $6$ $(=6+4-4)$.

\item {(Deter, Status Quo) or (Status Quo, Deter)}\\
The only deterrer gets a private benefit $(=6)$ and pays a public cost $(=4)$. So the payoff to the deterrer is $2$ $(=6-4)$. In case of adopting the status quo, it just costs $(=4)$ without any benefit.  The payoff for adopting the status quo is $-4$ $(=0-4)$.

\item {(Deter, Deter)}\\
The payoff to the players is $-2$ $(=6-4-4)$ because each can get a private benefit of 6 but they impose a public cost of 4 on each other. 
\end{itemize}

Notice that (Deter, Deter) is a pure Nash equilibrium because Deter is a dominant strategy for both players.  Yet both players receive higher payoffs from (Preempt, Preempt) and from (Status Quo, Status Quo), so this is a classic prisoner's dilemma situation.

Our aim here is to apply cooperative game theory to this model instead of non-cooperative game theory.  There are at least three kinds of solution in cooperative game theory, namely, the TU (transferable utility) solution, the NTU solution based on the Nash Bargaining Model, and the NTU solution based on the lambda transfer  approach.  Ferguson (2014) is recommended for background on this theory.

\subsection{TU solution}  
In a cooperative game with payoff bimatrix $(\bm A,\bm B)$, the players will agree to play so as to achieve $\sigma=\max_{i,j}(a_{ij}+b_{ij})$, and then will divide $\sigma$ between them in some way.  If the threat strategies are $\bm p$ for Player 1 and $\bm q$ for Player 2, Player 1 will accept no less than $D_1=\bm p^\T\bm A\bm q$ and player 2 will accept no less than $D_2=\bm p^\T\bm B\bm q$ since the players can receive them without agreement.  The players will negotiate about which point on the line segment $u+v=\sigma$ from $(D_1,\sigma-D_1)$ to $(\sigma-D_2,D_2)$ is the TU solution.  It should be the midpoint of the interval, i.e.,
\begin{equation*}
\bm{\varphi}=(\varphi_{1},\varphi_{2})=\bigg(\frac{\sigma+D_1-D_2}{2},\frac{\sigma-(D_1-D_2)}{2}\bigg).
\end{equation*}
This shows that Player 1 wants to maximize $D_1-D_2$, while Player 2 wants to minimize it.  Since $D_1-D_2=\bm p^\T(\bm A-\bm B)\bm q$, we see that the optimal threat strategies are given by the solution $(\bm p^*,\bm q^*)$ of the matrix game $\bm A-\bm B$.  With
$$
\delta=\text{Val}(\bm A-\bm B)=(\bm p^*)^\T(\bm A-\bm B)\bm q^*,
$$
the TU solution becomes 
\begin{equation*}
\bm\varphi^*=(\varphi_1^*,\varphi_2^*)=\bigg(\frac{\sigma+\delta}{2},\frac{\sigma-\delta}{2}\bigg).
\end{equation*}

Let 
\[
\setlength{\arraycolsep}{2mm}
\bm{A}=
\begin{pmatrix}2 & -2 & -6\\
4 & 0 & -4\\
6 & 2 & -2
\end{pmatrix}\quad\textrm{and}\quad
\bm{B}=
\begin{pmatrix}2 & 4 & 6\\
-2 & 0 & 2\\
-6 & -4 & -2
\end{pmatrix}
\]
as in \eqref{bimatrix}. Then the difference matrix 
\[
\setlength{\arraycolsep}{2mm}\bm{A}-\bm{B}=\begin{pmatrix}0 & -6 & -12\\
6 & 0 & -6\\
12 & 6 & 0
\end{pmatrix}
\]
has a saddle point at the lower right with value $\delta=0$. So $\bm{p}^{*}=(0,0,1)^\T$ and $\bm{q}^{*}=(0,0,1)^\T$ are the threat strategies and the disagreement point is $(D_1,D_2)=(-2,-2)$, which is the Nash equilibrium in the non-cooperative game. Also, we can get value $\sigma=\max(a_{ij}+b_{ij})=4$. Therefore the TU solution is 
\[
\bm\varphi^*=\bigg(\frac{\sigma+\delta}{2},\frac{\sigma-\delta}{2}\bigg)=\bigg(\frac{4+0}{2},\frac{4-0}{2}\bigg)=(2,2).
\]
Since the cooperative strategy gives (2,2), this does not require any side payment.

\subsection{NTU solution based on the Nash Bargaining Model}  
This model assumes that two elements should be given and known to the players.  One element is a compact (i.e., closed and bounded), convex set $S$
in the plane. We refer to $S$ as the NTU-feasible set. Another is a threat point, $(u^{*},v^{*})\in S$. Given an NTU-feasible set $S$ and a threat point $(u^{*},v^{*})\in S$, we can find a unique NTU solution $(\bar{u},\bar{v})\in S$ that maximizes $f(u,v)=(u-u^{*})(v-v^{*})$, as suggested by Nash.

\begin{thm} If there exists a point $(u,v)\in S$ with $u>u^{*}$
and $v>v^{*}$ then 
\begin{equation*}
\max_{u>u^{*},v>v^{*},(u,v)\in S}(u-u^{*})(v-v^{*})
\end{equation*}
is attained at a unique point $(\bar{u},\bar{v})$. 
\end{thm} 

\begin{proof} Suppose there are two different points $(u_1,v_1), (u_2,v_2)\in S$ that maximize $f(u,v)=(u-u^{*})(v-v^{*})$, and let $M$ be the maximum value. Since $M>0$, $u_1=u_2$ implies $v_1=v_2$. Since $S$ is convex and $(u,v)\in S$, without loss of generality we can suppose that $u_1<u_2$, in which case $v_1>v_2$, and put $(u,v)=\frac{1}{2}(u_1,v_1)+\frac{1}{2}(u_2,v_2)=\frac{1}{2}(u_1+u_2,v_1+v_2)$. Now 
\begin{align*}
f(u,v) & =  \bigg(\frac{u_1+u_2}{2}-u^{*}\bigg)\bigg(\frac{v_1+v_2}{2}-v^{*}\bigg)\\
 & =  \frac{(u_1-u^{*})+(u_2-u^{*})}{2}\cdot\frac{(v_1-v^{*})+(v_2-v^{*})}{2}\\
 & =  \frac{2(u_1-u^{*})(v_1-v^{*})-(u_1-u^{*})(v_1-v^{*})+2(u_2-u^{*})(v_2-v^{*})}{4}\\
 &   \quad{}+\frac{-(u_2-u^{*})(v_2-v^{*})+(u_1-u^{*})(v_2-v^{*})+(u_2-u^{*})(v_1-v^{*})}{4}\\
 & =  \bigg(\frac{(u_1-u^{*})(v_1-v^{*})}{2}+\frac{(u_2-u^{*})(v_2-v^{*})}{2}\bigg)+\frac{(u_1-u_2)(v_2-v_1)}{4}\\
 & =  M+\frac{(u_1-u_2)(v_2-v_1)}{4}.
\end{align*}
Since $u_1<u_2$ and $v_1>v_2$, the last fraction is positive, hence $f(u,v)>M$, which is a contradiction to the assumption that $M$ is the maximum value. Therefore, the point $(\bar{u},\bar{v})$ is unique. 
\end{proof}

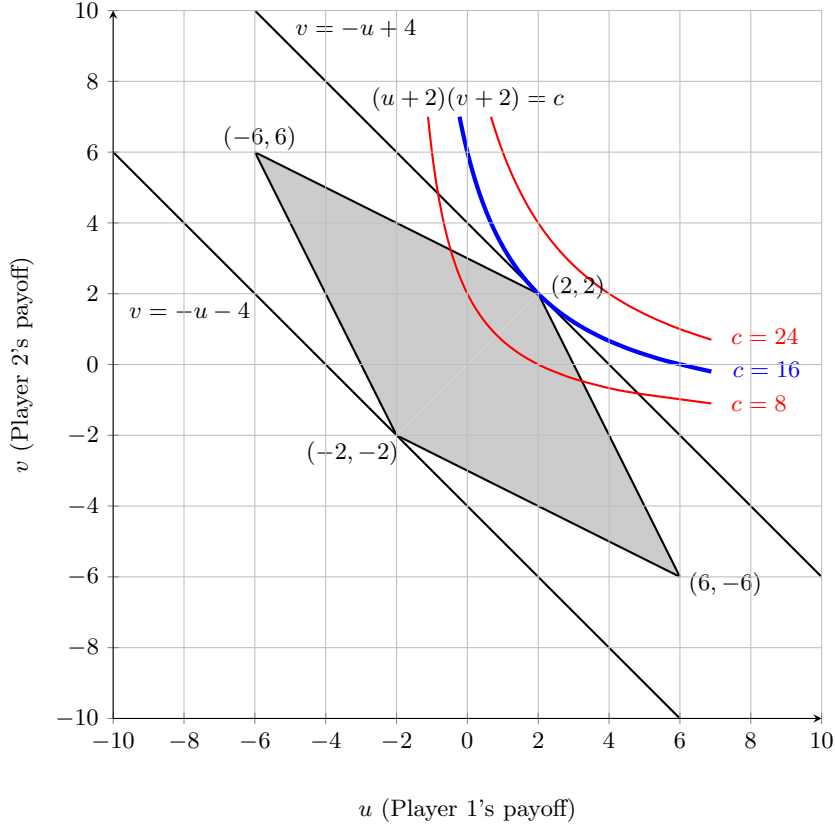
\begin{figure}[htb]  
\begin{center}
\begin{small}
 	\begin{tikzpicture}
 	\begin{axis}[
 	axis lines=middle,
 	axis line style= ->,
 	axis y line=left,
 	axis x line=bottom, 
 	axis on top=true,
 	xmin=-10,xmax=10,
 	ymin=-10,ymax=10,
 	x label style={at={(axis description cs:0.5,-0.1)},anchor=north},
 	y label style={at={(axis description cs:-0.1,.5)},rotate=90,anchor=south},
 	xlabel={$u$ (\text{Player 1's payoff})},
 	ylabel={$v$ (\text{Player 2's payoff})},
 	height=11cm,
 	width=11cm,
 	grid,
 	xtick={-10,-8,...,10},
 	ytick={-10,-8,...,10},
 	]
 	\addplot[name path=plot1,smooth,thick,black,-, domain=-6:10]({-x+4},{x});
 	\addplot[name path=plot2,smooth, thick,black,-,domain=-10:6]({-x-4},{x});
 	\addplot[name path=plot3,smooth, thick,black,-,domain=-6:2]({(-0.5)*x+3},{x});
 	\addplot[name path=plot4,smooth, thick,black,-,domain=2:6]({(-2)*x+6},{x});
 	\addplot[name path=plot5,smooth, thick,black,-,domain=-2:6]({(-0.5)*x-3},{x});
 	\addplot[name path=plot6,smooth, thick,black,-,domain=-6:-2]({(-2)*x-6},{x});
 	\addplot[gray!40!white] fill between[of= plot3 and plot6, soft clip={domain=-10:10}];
 	\addplot[gray!40!white] fill between[of= plot4 and plot5, soft clip={domain=-10:10}];
 	\addplot[name path=plot7,smooth, ultra thick,blue,-,domain=-0.2:7]({(16)/(x+2)-2},{x});
 	\addplot[name path=plot8,smooth, thick,red,-,domain=0.7:7]({(24)/(x+2)-2},{x});
 	\addplot[name path=plot8,smooth, thick,red,-,domain=-1.1:7]({(8)/(x+2)-2},{x});
 	\node [right, black,ultra thick] at (axis cs: -3,7.5) {$(u+2)(v+2) = c$};
 	\node [right, blue,ultra thick] at (axis cs: 7.2,-.15) {$c=16$};
 	\node [right, red] at (axis cs: 7.2,.8) {$c=24$};
 	\node [right, red] at (axis cs: 7.2,-1.1) {$c=8$};
 	\node [right, black] at (axis cs: -5.1,9.5) {$v=-u+4$};
 	\node [left, black] at (axis cs: -5.9,1.5) {$v=-u-4$};
 	\node [right, black] at (axis cs: 2.1,2.2) {$(2,2)$};
 	\node [right, black] at (axis cs: 6.0,-6.2) {$(6,-6)$};
 	\node [right, black] at (axis cs: -7.15,6.4) {$(-6,6)$};
 	\node [right, black] at (axis cs: -4.8,-2.5) {$(-2,-2)$};
 	\end{axis};
 	
 	\end{tikzpicture}   
\protect\caption{\label{feasible set}TU feasible set: $-u-4\leq v\leq-u+4$. NTU feasible set: shaded area
of rhombus.}
\end{small}
\end{center}
\end{figure}

We can show our bimatrix geometrically in Figure \ref{feasible set}. First, we consider the disagreement point $(u^{*},v^{*})=(-2,-2)$ in the TU solution section as the threat point. The set of Pareto optimal points consists of the two line segments from $(-6,6)$ to $(2,2)$ and from $(2,2)$
to $(6,-6)$. The NTU solution is that point along this path which maximizes $(u+2)(v+2)$. Let $f(u)=(u+2)(v+2)$. Now, the line segment from $(-6,6)$ to $(2,2)$ has the equation, $v=-\frac{1}{2}u+3$.  So we can rewrite $f(u)=(u+2)(v+2)=(u+2)(-\frac{1}{2}u+5)=-\frac{1}{2}u^{2}+4u+10$.  It has its maximum in $u\in[-6,2]$ at $u=2$ where $v$ has the value 2. Similarly, the line segment from $(2,2)$ to $(6,-6)$ satisfies the equation $v=-2u+6$. So we can write $f(u)=-2u^{2}+4u+16$. In this case, it has its maximum in $u\in[2,6]$ at $u=2$ and $v=2$ too. Hence, $f(u)=(u+2)(v+2)$ is maximized along the Pareto boundary at $(\bar{u},\bar{v})=(2,2)$ which is the NTU solution of our example.

\subsection{NTU solution based on the lambda transfer approach} 

If the original bimatrix $(\bm{A},\bm{B})$ and its utilities are not measured in the same units, we can change it into a bimatrix to which the TU theory applies. If an increase of one unit in Player 1's utility is worth an increase $\lambda$ $(>0)$ units in Player 2's utility, then the bimatrix $(\lambda\bm{A},\bm{B})$ has transferable utility. By the TU-solution method with bimatrix $(\lambda\bm{A},\bm{B})$,
the lambda transfer solution for the NTU game is 
\begin{equation}\label{varphi(lambda)}
\bm{\varphi}(\lambda)=(\varphi_{1}(\lambda),\varphi_{2}(\lambda))=\bigg(\frac{\sigma(\lambda)+\delta(\lambda)}{2\lambda},\frac{\sigma(\lambda)-\delta(\lambda)}{2}\bigg),
\end{equation}
where $\sigma(\lambda)=\max_{i,j}(\lambda a_{ij}+b_{ij})$
and $\delta(\lambda)=\mathrm{Val}(\lambda\bm{A}-\bm{B})=\bm{p}^{*\T}(\lambda\bm{A}-\bm{B})\bm{q}^{*}$.
Generally, there is a unique $\lambda$, denoted by $\lambda^*$, such that \eqref{varphi(lambda)} is on the Pareto optimal boundary of the NTU feasible set.  Then $\bm{\varphi}(\lambda^*)$ is the NTU solution.

In our example we have the transferred bimatrix
\[
(\lambda\bm{A},\bm{B})=\begin{pmatrix}(2\lambda,2) & (-2\lambda,4) & (-6\lambda,6)\\
(4\lambda,-2) & (0,0) & (-4\lambda,2)\\
(6\lambda,-6) & (2\lambda,-4) & (-2\lambda,-2)
\end{pmatrix}.
\]
 Then $\delta(\lambda)=-2\lambda+2$ can be found easily through the
difference matrix 
\[
\setlength{\arraycolsep}{2mm}\lambda\bm{A}-\bm{B}=\begin{pmatrix}2\lambda-2 & -2\lambda-4 & -6\lambda-6\\
4\lambda+2 & 0 & -4\lambda-2\\
6\lambda+6 & 2\lambda+4 & -2\lambda+2
\end{pmatrix},
\]
which has a saddle point at the lower right.
It is easy to check that $\sigma(\lambda)=\max_{i,j}(\lambda a_{ij}+b_{ij})$
is given by
\begin{equation*}
\sigma(\lambda)=\begin{cases}
-6\lambda+6 & \text{if \ensuremath{0<\lambda\leq\frac{1}{2}}}\\
2\lambda+2 & \text{if \ensuremath{\frac{1}{2}\leq\lambda\leq2}}\\
6\lambda-6 & \text{if \ensuremath{\lambda\geq2}}.
\end{cases}
\end{equation*}

\begin{description}
\item [{Case}] 1: $0<\lambda\leq\frac{1}{2}$ 
\end{description}
The candidate of the solution is 
\begin{align*}
\bm{\varphi}(\lambda) & = \bigg(\frac{\sigma(\lambda)+\delta(\lambda)}{2\lambda},\frac{\sigma(\lambda)-\delta(\lambda)}{2}\bigg)\\
 & = \bigg(\frac{(-6\lambda+6)+(-2\lambda+2)}{2\lambda},\frac{(-6\lambda+6)-(-2\lambda+2)}{2}\bigg)\\
 & = \bigg(\!\!-4+\frac{4}{\lambda},-2\lambda+2\bigg),
\end{align*}
which does not intersect the NTU feasible set. \begin{description}
\item [{Case}] 2: $\frac{1}{2}\leq\lambda\leq2$ 
\end{description}
Since $\sigma(\lambda)=2\lambda+2$ and $\delta=-2\lambda+2$, we get $\bm{\varphi}(\lambda)=(2/\lambda,2\lambda)$ as the solution.  Only the point $\bm{\varphi}(1)=(2,2)$ belongs to the NTU feasible set.

\begin{description}
\item [{Case}] 3: $\lambda\geq2$ 
\end{description}
The final step is to check whether $\bm{\varphi}(\lambda)=(2-2/\lambda,4\lambda-4)$ is a possible solution, and it is not.\medskip{}

From the cases above, our final NTU solution through the lambda transfer approach is $\bm{\varphi}(\lambda^{*})=(2,2)$ at $\lambda^{*}=1$.

We conclude that all three approaches lead to the same solution, namely (Preempt, Preempt), in contrast to the non-cooperative (Nash equilibrium) solution, (Deter, Deter).


\section{Generalization}

The bimatrix \eqref{bimatrix} was a very specific symmetric example, which we now want to generalize.  The bimatrix 
\begin{small}
\begin{equation}\label{bimatrix-general}
\bordermatrix{
& \text{Preempt} & \text{Status Quo} & \text{Deter} \cr
\text{Preempt} & (2B-c,2B-c) &  (B-c,B) & (B-c-C,B+b-C)\cr
\text{Status Quo} & (B,B-c) & (0,0) & (-C,b-C)\cr
\text{Deter} & (B+b-C,B-c-C) & (b-C,-C) & (b-2C,b-2C)}
\end{equation}
\end{small}
from Arce M. and Sandler (2005) shows the generalized payoffs.  As before, the row player is Player 1 (e.g., the US) and the column player is Player 2 (e.g., the EU), with $B$ and $c$ representing the public benefit and the private cost when a player uses the preemption policy, and $b$ and $C$ denoting
the private benefit and the public cost when a player takes the deterrence action.  Here $B<c<2B$ and $C<b<2C$ are assumed.  The derivation of \eqref{bimatrix-general} is similar to that of \eqref{bimatrix}.

To make the game easier to analyze, we make additional assumptions beyond those of Arce M. and Sandler (2005).  We assume that $B=C$, $c=\alpha B$, and $b=\beta C$, where $1<\alpha,\beta<2$, on the basis of \eqref{bimatrix} and \eqref{bimatrix-general}.  This reduces \eqref{bimatrix-general}, after factoring out $B$, to
\begin{small}
\begin{equation*}
\bordermatrix{
& \text{Preempt} & \text{Status Quo} & \text{Deter}\cr
\text{Preempt} & (2-\alpha,2-\alpha) & (-(\alpha-1),1) & (-\alpha,\beta)\cr
\text{Status Quo} & (1,-(\alpha-1)) & (0,0) & (-1,\beta-1)\cr
\text{Deter} & (\beta,-\alpha) & (\beta-1,-1) & (-(2-\beta),-(2-\beta))}=(\bm{U},\bm{V}),
\end{equation*}
\end{small}
a matrix with two parameters instead of four.

\subsection{TU solution} 
Since 
\[
\setlength{\arraycolsep}{1mm}\bm{U}=\begin{pmatrix}2-\alpha & -(\alpha-1) & -\alpha\\
1 & 0 & -1\\
\beta & \beta-1 & -(2-\beta)
\end{pmatrix}\quad\textrm{and}\quad\bm{V}=\begin{pmatrix}2-\alpha & 1 & \beta\\
-(\alpha-1) & 0 & \beta-1\\
-\alpha & -1 & -(2-\beta)
\end{pmatrix},
\]
 the difference matrix 
\[
\setlength{\arraycolsep}{2mm}\bm{U}-\bm{V}=\begin{pmatrix}0 & -\alpha & -\alpha-\beta\\
\alpha & 0 & -\beta\\
\alpha+\beta & \beta & 0
\end{pmatrix}
\]
has a saddle point at the lower right with value $\delta=0$. So $\bm{p}^{*}=(0,0,1)^\T$ and $\bm{q}^{*}=(0,0,1)^\T$ are the threat strategies and the disagreement point is $(-(2-\beta),-(2-\beta))$, which is the Nash equilibrium in the non-cooperative game. Also, we can get 
\[
\setlength{\arraycolsep}{2mm}\bm{U}+\bm{V}=\begin{pmatrix}2(2-\alpha) & 2-\alpha & -\alpha+\beta\\
2-\alpha & 0 & -(2-\beta)\\
-\alpha+\beta & -(2-\beta) & -2(2-\beta)
\end{pmatrix}.
\]
We see that $\sigma=\max_{i,j}(u_{ij}+v_{ij})=2(2-\alpha)$ because $2(2-\alpha)-(-\alpha+\beta)=(2-\alpha)+(2-\beta)>0$ under the condition $1<\alpha,\beta<2$.  Now the TU solution is 
\[
\bm{\varphi}^{*}=\bigg(\frac{\sigma+\delta}{2},\frac{\sigma-\delta}{2}\bigg)=(2-\alpha,2-\alpha).
\]
Since the cooperative strategy gives $(u_{11},v_{11})=(2-\alpha,2-\alpha)$, this does not require any side payment. Converting it to the original notation, we get $\bm{\varphi}^{*}=(2B-c,2B-c)$.

\subsection{NTU solution based on the Nash Bargaining Model} 

First of all, we have to compare the slopes of the line segments representing the Pareto optimal boundary to find the NTU solution because the slopes
could depend on the parameters $\alpha$ and $\beta$. We can think of two cases as in Figure~\ref{NTU set1} and Figure~\ref{NTU set2}. Figure~\ref{NTU set1} shows that the slope of the line segment $P_1$ from $(-\alpha,\beta)$ to $(2-\alpha,2-\alpha)$ is less than that of the line segment $Q_1$ from $(-(\alpha-1),1)$ to $(2-\alpha,2-\alpha)$, equivalently, the slope of the line segment $P_2$ from $(2-\alpha,2-\alpha)$ to $(\beta,-\alpha)$ is greater than that of the line segment $Q_2$ from $(2-\alpha,2-\alpha)$ to $(1,-(\alpha-1))$, and vice versa in Figure~\ref{NTU set2}.
\begin{description}
\item [{Case}] 1: $\text{slope}(P_1)<\text{slope}(Q_1)$, i.e., $-(\alpha+\beta-2)/2<-(\alpha-1)$ or $\alpha<\beta$; equivalently, $\text{slope}(P_2)>\text{slope}(Q_2)$.
\end{description}

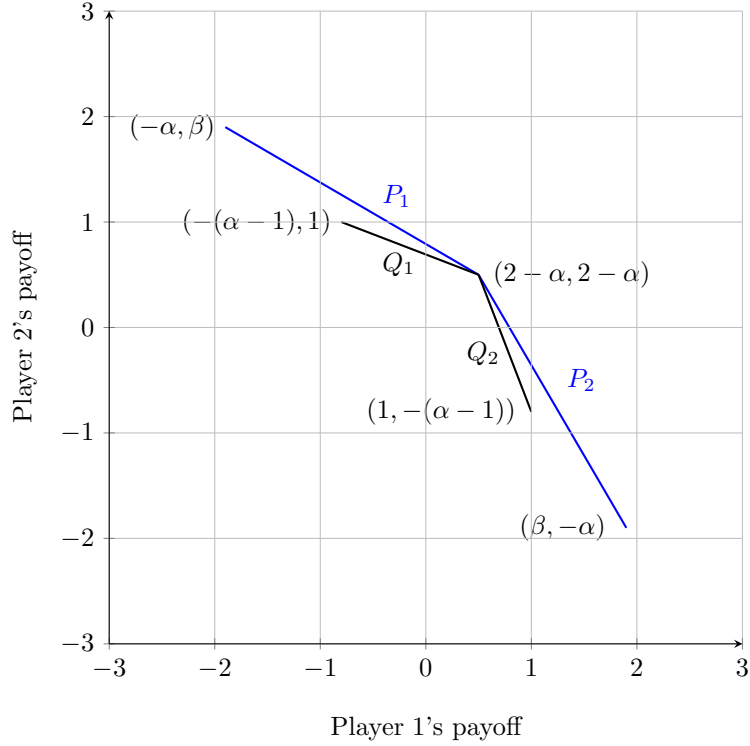
\begin{figure}[htb]
\begin{center}
 \begin{tikzpicture}
 \begin{axis}[
 axis lines=middle,
 axis line style= {-},
 axis y line=left,
 axis x line=bottom, 
 axis on top=true,
 xmin=-3,xmax=3,
 ymin=-3,ymax=3,
 x label style={at={(axis description cs:0.5,-0.1)},anchor=north},
 y label style={at={(axis description cs:-0.1,.5)},rotate=90,anchor=south},
 xlabel={\text{Player 1's payoff}},
 ylabel={\text{Player 2's payoff}},
 height=10cm,
 width=10cm,
 grid,
 xtick={-3,-2,...,3},
 ytick={-3,-2,...,3},
 ]
 \addplot[name path=plot1,smooth,thick,blue,-, domain=0.5:1.9]({-(12/7)*x+19/14},{x});
 \addplot[name path=plot2,smooth,thick,blue,-, domain=-1.9:0.5]({-(7/12)*x+19/24},{x});
 \addplot[name path=plot3,smooth,thick,black,-, domain=0.5:1]({-(13/5)*x+9/5},{x});
 \addplot[name path=plot4,smooth,thick,black,-, domain=-0.8:0.5]({-(5/13)*x+9/13},{x});
 \node [right, black] at (axis cs: -2.9,1.9) {$(-\alpha, \beta)$};
 \node [right, black] at (axis cs: 0.55,0.5) {$(2-\alpha, 2-\alpha)$};
 \node [right, black] at (axis cs: 0.8,-1.9) {$(\beta,-\alpha)$};
 \node [right, black] at (axis cs: -2.4,01.0) {$(-(\alpha-1), 1)$};
 \node [right, black] at (axis cs: -0.65,-0.8) {$(1,-(\alpha-1))$};
 
 \node [right, blue] at (axis cs: -0.5,1.25) {$P_1$};
 \node [right, blue] at (axis cs: 1.25,-0.5) {$P_2$};
 \node [right, black] at (axis cs: -0.5,0.6) {$Q_1$};
 \node [right, black] at (axis cs: 0.3,-0.25) {$Q_2$};
 \end{axis};
 
 \end{tikzpicture}   
 \protect\caption{\label{NTU set1} The Pareto optimal boundary in Case 1 is shown in blue.}
 \end{center}
\end{figure}

Let us start with Figure \ref{NTU set1}. In this case, the NTU solution should be on the line of the equation 
\begin{equation}\label{P1}
v=-\frac{\alpha+\beta-2}{2}u+\frac{(2-\alpha)(\alpha+\beta)}{2},\quad -\alpha\le u\le 2-\alpha,
\end{equation}
or that of the equation
\begin{equation}\label{P2}
v=-\frac{2}{\alpha+\beta-2}u+\frac{(2-\alpha)(\alpha+\beta)}{\alpha+\beta-2},\quad 2-\alpha\le u\le \beta.
\end{equation}

We consider the disagreement point $(u^{*},v^{*})=(-(2-\beta),-(2-\beta))$ in the TU solution section as the threat point. The set of Pareto optimal points consists of the two line segments \eqref{P1} and \eqref{P2} above.  The NTU solution is that point $(u,v)$ along this path that maximizes $(u+2-\beta)(v+2-\beta)$. Now, using the equation \eqref{P1}, we can rewrite this as a quadratic
\[
f(u):=(u+2-\beta)\bigg(\!\!-\frac{\alpha+\beta-2}{2}u+\frac{(2-\alpha)(\alpha+\beta)}{2}+2-\beta\bigg).
\]
The maximum of $f(u)$ occurs at 
\[
\hat{u}=\frac{-\alpha^{2}+\beta^{2}-4\beta+8}{2(\alpha+\beta-2)},
\]
but $\hat u-(2-\alpha)=(4-\alpha-\beta)^2/[2(\alpha+\beta-2)]>0$, so the maximum of $f(u)$ over $-\alpha\le u\le2-\alpha$ occurs at $u=2-\alpha$.
Similarly, if we substitute the linear function \eqref{P2} for $v$ in $(u+2-\beta)(v+2-\beta)$, a similar argument shows that $f(u)$ is maximized over $[2-\alpha,\beta]$ at $u=2-\alpha$.  Hence, $(u+2-\beta)(v+2-\beta)$ is maximized along the Pareto optimal boundary at $(\bar{u},\bar{v})=(2-\alpha,2-\alpha)$, which is the NTU solution.

\begin{description}
\item [{Case}] 2:  $\text{slope}(P_1)>\text{slope}(Q_1)$, i.e., $-(\alpha+\beta-2)/2>-(\alpha-1)$ or $\alpha>\beta$; equivalently, $\text{slope}(P_2)<\text{slope}(Q_2)$. 
\end{description}
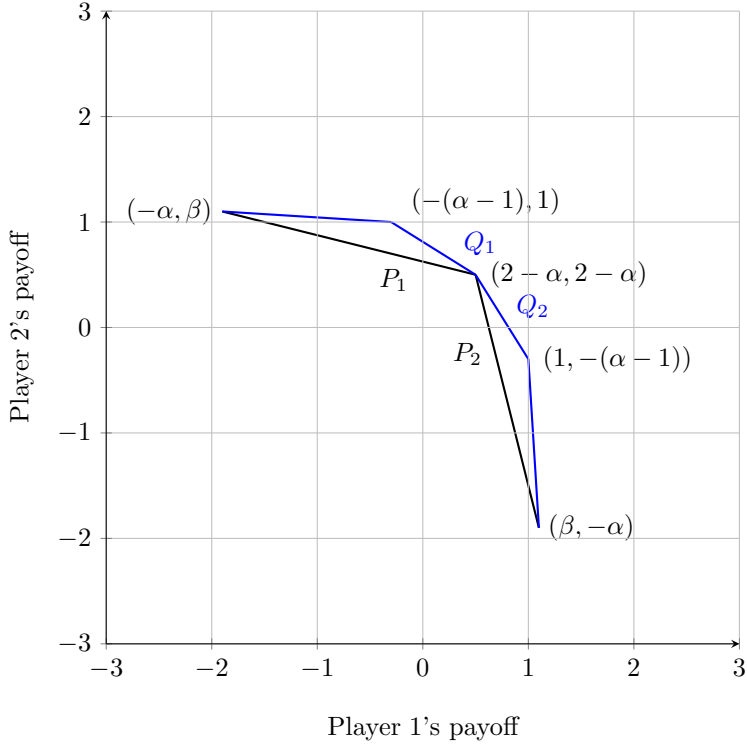
\begin{figure}[htb]
 \begin{center}		
 	\begin{tikzpicture}
 	\begin{axis}[
 	axis lines=middle,
 	axis line style= {-},
 	axis y line=left,
 	axis x line=bottom, 
 	axis on top=true,
 	xmin=-3,xmax=3,
 	ymin=-3,ymax=3,
 	x label style={at={(axis description cs:0.5,-0.1)},anchor=north},
 	y label style={at={(axis description cs:-0.1,.5)},rotate=90,anchor=south},
 	xlabel={\text{Player 1's payoff}},
 	ylabel={\text{Player 2's payoff}},
 	height=10cm,
 	width=10cm,
 	grid,
 	xtick={-3,-2,...,3},
 	ytick={-3,-2,...,3},
 	]
 	\addplot[name path=plot1,smooth,thick,black,-, domain=0.5:1.1]({-(4)*x+2.5},{x});
 	\addplot[name path=plot2,smooth,thick,black,-, domain=-1.9:0.5]({-(1/4)*x+5/8},{x});
 	\addplot[name path=plot3,smooth,thick,blue,-, domain=1:1.1]({-(16)*x+15.7},{x});
 	\addplot[name path=plot4,smooth,thick,blue,-, domain=-1.9:-0.3]({-(1/16)*x+15.7/16},{x});	
 	\addplot[name path=plot5,smooth,thick,blue,-, domain=0.5:1]({-(8/5)*x+13/10},{x});
 	\addplot[name path=plot6,smooth,thick,blue,-, domain=-0.3:0.5]({-(5/8)*x+13/16},{x});		
 	\node [right, black] at (axis cs: -2.9,1.1) {$(-\alpha, \beta)$};
 	\node [right, black] at (axis cs: -0.2,1.2) {$(-(\alpha-1), 1)$};
 	\node [right, black] at (axis cs: 0.55,0.5) {$(2-\alpha, 2-\alpha)$};
 	\node [right, black] at (axis cs: 1.05,-0.3) {$(1,-(\alpha-1))$};
 	\node [right, black] at (axis cs: 1.1,-1.9) {$(\beta,-\alpha)$};

 \node [right, blue] at (axis cs: 0.3,0.8) {$Q_1$};
 \node [right, blue] at (axis cs: 0.8,0.2) {$Q_2$};
 \node [right, black] at (axis cs: -0.5,0.45) {$P_1$};
 \node [right, black] at (axis cs: 0.2,-0.25) {$P_2$};
 	\end{axis};
 	\end{tikzpicture}
 \end{center} 
\protect\caption{\label{NTU set2}The Pareto optimal boundary in Case 2 is shown in blue.}
\end{figure}

We should be more careful with this case because the constraint $\alpha>\beta$ implies that the Pareto optimal boundary comprises four different line segments (see Figure \ref{NTU set2}). Now, we consider the two (unlabeled) outer line segments from $(-\alpha,\beta)$ to $(-(\alpha-1),1)$ and from $(1,-(\alpha-1))$ to $(\beta,-\alpha)$, whose equations are 
\begin{equation}\label{Q1*}
v=-(\beta-1)u+\alpha+\beta-\alpha\beta,\quad -\alpha\le u\le -(\alpha-1),
\end{equation}
 and 
\begin{equation}\label{Q2*}
v=-\frac{1}{\beta-1}u+\frac{\alpha+\beta-\alpha\beta}{\beta-1},\quad 1\le u\le \beta.
\end{equation}

First, we check whether the NTU solution could be on the line of equation \eqref{Q1*} or \eqref{Q2*}. If we maximize $(u+2-\beta)(v+2-\beta)$ along \eqref{Q1*}, then we must maximize $f(u):=(u+2-\beta)(-(\beta-1)u+\alpha-\alpha\beta+2)$ over $-\alpha\le u\le-(\alpha-1)$.  We find that the maximum of this quadratic occurs at $\hat u>-(\alpha-1)$, so the maximum over $[-\alpha,-(\alpha-1)]$ occurs at $u=-(\alpha-1)$. Similarly, maximizing along \eqref{Q2*}, we must maximize $f(u):=(u+2-\beta)(-(\beta-1)^{-1}u+(\beta-1)^{-1}(\alpha+\beta-\alpha\beta)+2-\beta)$ over $1\le u\le\beta$.  The maximum occurs at $u=1$.

Now, we focus on two other line segments, $Q_1$ given by
\begin{equation}\label{Q1}
v=-(\alpha-1)u+\alpha(2-\alpha),\quad -(\alpha-1)\le u\le 2-\alpha,
\end{equation}
and $Q_2$ given by
\begin{equation*}
v=-\frac{1}{\alpha-1}u+\frac{\alpha(2-\alpha)}{\alpha-1},\quad 2-\alpha\le u\le 1.
\end{equation*}

Along the line segment \eqref{Q1}, we can maximize $(u-\beta+2)(v-\beta+2)$ by maximizing $f(u):=(u-\beta+2)(-(\alpha-1)u+\alpha(2-\alpha)+2-\beta)$ over $-(\alpha-1)\le u\le2-\alpha$.  The maximum of the quadratic occurs at $\hat u>2-\alpha$, so its maximum over $[-(\alpha-1),2-\alpha]$ occurs at $u=2-\alpha$, and $f(2-\alpha)=(4-\alpha-\beta)^2$.  Similarly, the maximum of $(u-\beta+2)(v-\beta+2)$ over $2-\alpha\le u\le1$ occurs at $u=2-\alpha$ with the same result. Hence, $f(u)=(u-\beta+2)(v-\beta+2)$ is maximized along the Pareto optimal boundary at $(\bar{u},\bar{v})=(2-\alpha,2-\alpha)$, which is the general NTU solution when $\alpha>\beta$.  This coincides with our result in the case $\alpha<\beta$, and both arguments apply when $\alpha=\beta$.

\subsection{NTU solution based on the lambda transfer approach}

We have the transferred bimatrix, 
\[
(\lambda\bm{U},\bm{V})=\begin{pmatrix}(\lambda(2-\alpha),2-\alpha) & (-\lambda(\alpha-1),1) & (-\lambda\alpha,\beta)\\
(\lambda,-(\alpha-1)) & (0,0) & (-\lambda,\beta-1)\\
(\lambda\beta,-\alpha) & (\lambda(\beta-1),-1) & (-\lambda(2-\beta),-(2-\beta))
\end{pmatrix}.
\]
and therefore
\begin{align*}
\lambda\bm{U}-\bm{V} & = \begin{pmatrix}(\lambda-1)(2-\alpha) & -\lambda(\alpha-1)-1 & -\lambda\alpha-\beta\\
\lambda+\alpha-1 & 0 & -\lambda-\beta+1\\
\lambda\beta+\alpha & \lambda(\beta-1)+1 & (1-\lambda)(2-\beta)
\end{pmatrix}.
\end{align*}
We can verify that the $(3,3)$ entry is a saddle point, so $\delta(\lambda)=(1-\lambda)(2-\beta)$.  This involves showing that $(1-\lambda)(2-\beta)$ is a row minimum and a column maximum, regardless of $0<\lambda<\infty$.

To evaluate $\sigma(\lambda)$ we need the maximal entry of
\[
\setlength{\arraycolsep}{2mm}\lambda\bm{U}+\bm{V}=\begin{pmatrix}(\lambda+1)(2-\alpha) & -\lambda(\alpha-1)+1 & -\lambda\alpha+\beta\\
\lambda-\alpha+1 & 0 & -\lambda+\beta-1\\
\lambda\beta-\alpha & \lambda(\beta-1)-1 & -(\lambda+1)(2-\beta)
\end{pmatrix},
\]
so let us first consider the case $0<\lambda\le1$.  Then, comparing $(\lambda+1)(2-\alpha)$ with each of the other entries of $\lambda\bm{U}+\bm{V}$, we find that $\sigma(\lambda)=(\lambda+1)(2-\alpha)$ provided 
$$
\frac{\alpha+\max(\alpha,\beta)}{2}-1\le\lambda\le1.
$$
(We are using the fact that $\alpha>\beta/(3-\alpha)$, regardless of $\alpha,\beta\in(1,2)$.)
In this case, the TU solution of the transferred problem is
\begin{align*}
\bm\varphi(\lambda)&=\bigg(\frac{\sigma(\lambda)+\delta(\lambda)}{2\lambda},\frac{\sigma(\lambda)-\delta(\lambda)}{2}\bigg)\\
&=\bigg(\frac{\beta-\alpha}{2}+\frac{4-\alpha-\beta}{2\lambda},\frac{\beta-\alpha}{2}+\frac{(4-\alpha-\beta)\lambda}{2}\bigg),
\end{align*}
which reduces to $(2-\alpha,2-\alpha)$ when $\lambda=1$.  Thus, $\lambda^*=1$ and $\bm\varphi(\lambda^*)=(2-\alpha,2-\alpha)$.  A similar argument applies when $1\le\lambda<\infty$.

\section{Conclusion}

Using cooperative game theory, we obtained a different solution than the one found using non-cooperative game theory. Our game solution against terrorism is to take a firm attitude toward terrorists, that is, (preempt, preempt), even though there are many constraints in the real world.  Arce M. and Sandler (2005) wanted to show why countries facing terrorism take the passive action against terrorists. In contrast, this paper shows there is a positive effect when all countries facing terrorism stand firm, cooperating with each other.


\begin{thebibliography}{0}

\bibitem{AS}{Daniel G. Arce M. and Todd Sandler (2005) Counterterrorism: A Game-Theoretic Analysis. \textit{The Journal of Conflict Resolution},
 \textbf{49} (2), The Political Economy of Transnational Terrorism, pp. 183--200.}

\bibitem{F}{Thomas S. Ferguson (2014) \textit{Game Theory}, Second Edition. \url{http://www.math.ucla.edu/tom/GameTheory/Contents.html}.}

\end{thebibliography}
\end{document}